\documentclass[reqno]{siamart220329}

\usepackage{amsfonts,amsmath,amssymb}
\usepackage{algorithm,algpseudocode}
\newtheorem*{observation}{Observation}

\newsiamremark{remark}{Remark}
\headers{On the growth of nonconvex functionals}{A. Dom\'inguez Corella and T. Minh L\^e}

\title{On the growth of nonconvex functionals at strict  local minimizers\thanks{\textbf{Funding.} The first author was supported by the Alexander von Humboldt Foundation.}}

\author{Alberto Dom\'inguez Corella\thanks{Institut de Mathématiques de Jussieu - Paris Rive Gauche, Campus Pierre and Marie Curie, Sorbonne Universit\'e, 75005 Paris, France; \tt \email{alberto.of.sonora@gmail.com}}
	\and Tr\'i Minh L\^e\thanks{Institut f\"{u}r Stochastik und Wirtschaftsmathematik, VADOR E105-04, TU Wien, Wiedner Hauptstra{\ss }e 8, A-1040 Wien, \"Osterreich; \tt \email{minh.le@tuwien.ac.at}}}

\newcommand{\R}{\mathbb{R}}
\newcommand{\N}{\mathbb{N}}
\newcommand{\B}{\mathbb{B}}
\newcommand{\dom}{\mathrm{dom}}
\newcommand{\conv}{\overline{\mathrm{conv}}}
\newcommand{\im}{\mathrm{Im}}

\newcommand{\Lip}{\mathrm{Lip}}

\DeclareMathOperator*{\argmin}{arg\,min}

\begin{document}
	\maketitle
	\begin{abstract}
		We give new characterizations of growth conditions at strict local minimizers. The main characterizations are a variant of the so-called tilt stability property and an analog of the classical Polyak--\L{}ojasiewicz condition, where the gradient is replaced by linear perturbations. 
	\end{abstract}
	
	\begin{keywords}
		growth condition, nonconvex, subregularity, tilt stability, Polyak-\L ojasiewicz
	\end{keywords}
	
	\begin{MSCcodes}
		49J52, 49K40, 90C31, 90C48.
	\end{MSCcodes}

	
	\section{Overview of results and literature review} 
	Due to their crucial role in error estimates, growth conditions are a fundamental part of optimization. They are present in the analysis of discretization schemes, convergence rates, stability, propagation of errors, etc. 
	Their conceptual understanding  delineates the limits among assumptions that can be suppressed or must be maintained to achieve specific error estimates. This manuscript provides a couple of new and maybe unexpected equivalent conditions for the growth of a functional at a strict local minimizer. 
	
	\subsection{Overview of results}
	Let  $\big(X,\|{\cdot}\|\big)$ be a reflexive space over the real numbers and consider a proper lower semicontinuous function $f: X \rightarrow \R \cup \{ + \infty \}$.
	\smallbreak
	We assume that $f$ has a strict local minimizer $\bar x \in X$ and fix   $\delta>0$ such that 
	\[
	x \in \B\big(\bar x,\delta\big) \setminus \{\bar x\}\quad \implies \quad 	f(\bar x) < f(x).
	\]	
	In order to motivate the main characterization in the manuscript, let us consider the following observation illustrating a technique to  build sequences of local solutions to linearly perturbed problems. 
	\begin{observation}
		Let $\{x_n\}_{n\in\mathbb N}\subset X$ and $\{\xi_n\}_{n\in\mathbb N}\subset X^*$ be sequences such that 
		\begin{align*}
			x_n\in \argmin_{y\in \B(\bar x, \delta)} \big\{f(y) - \langle \xi_n, y\rangle\big\}\quad \forall n\in\mathbb N. 
		\end{align*}
		If $x_n\longrightarrow\bar x$ as $n\longrightarrow+\infty$, then $x_n$ is a local minimizer of $f-\langle \xi_n, \cdot\rangle$ for all $n\in\mathbb N$ sufficiently large. 
	\end{observation}	
	The technique relies on obtaining minimizers over a closed ball and ensuring their convergence. One usually employs polynomial growth conditions on the functional (such as (\ref{growth.condi}) below) to establish both convergence of minimizers and error estimates when perturbations tend to zero.
	A natural question is whether growth conditions can be avoided while still deriving stability estimates (such as (\ref{subreg}) and (\ref{PL.ineq})  below) that allow to obtain the same conclusions. 
	The answer is negative; a polynomial growth condition is necessary, at least when seeking Hölder error bounds.
	\smallbreak 
	Let us now state  our main result. 
	\begin{theorem}\label{thm.growth}
		Let $p, q \in (1, + \infty)$ be such that $p^{-1} + q^{-1}=1$. 
		The following statements are equivalent.
		\begin{itemize}
			\item[(i)] (Growth condition) There exists $\gamma>0$ such that 
			\begin{equation}\label{growth.condi}
				x\in \B\big(\bar x, \delta \big)\quad \implies\quad	{f}(x)\ge{f}(\bar x)+\gamma \|x-\bar x\|^p.
			\end{equation}
			
			\item[(ii)] (Tilt sub-stability) There exists $\kappa > 0$ such that, for any $\xi \in X^*$, 
			\begin{equation}\label{subreg}
				x\in \argmin_{y\in \B(\bar x, \delta)} \big\{f(y) - \langle \xi, y\rangle\big\} \quad \implies \quad\|x-\bar x\|\le\kappa \|\xi\|^{\frac{q}{p}}.
			\end{equation}	
			
			\item[(iii)]  (\L{}ojasiewicz--type inequality) There exists $\mu>0$ such that, for any $\xi\in X^*$, 
			\begin{equation}\label{PL.ineq}
				x\in \argmin_{y\in \B (\bar x, \delta)} \big\{f(y) - \langle \xi, y\rangle\big\} \quad \implies \quad f(x) - f(\bar x)\le\mu \|\xi\|^q.
			\end{equation}		
			
		\end{itemize}
	\end{theorem} 
	\begin{remark}
		Since $\delta>0$ can be arbitrarily chosen among the radii of strict minimality, the results extend globally for functions with bounded domain. This is particularly useful for constrained optimization problems over closed bounded sets. For such an application, we refer to Section~\ref{elliptic}, where an elliptic tracking problem is studied; there, we establish a novel equivalence between second-order conditions and the sensitivity of solutions with respect to data perturbations.
	\end{remark}
	
	\begin{remark}\label{rmk.const-rela}
		The relation between constants in Theorem \ref{thm.growth} can be described as follows. 
		Implication $(i)\implies (ii)$ with $\kappa=\gamma^{-\frac{q}{p}}$; implication  $(ii)\implies (iii)$ with $\mu = \kappa$; implication  $(iii)\implies (i)$ with $\gamma = p^{-q}\mu^{-1}$. 
	\end{remark}
	
	\begin{remark}
		The results in this section extend to a slightly more general setting, where $X$ is a Banach space and $f: X \to \mathbb{R} \cup \{+\infty\}$ is a proper lower semicontinuous function such that $\conv(\dom f)$ has the Radon--Nikodým property. For more details and further extensions, we refer to Section~\ref{further}.
	\end{remark}
	
	An immediate consequence of the previous theorem is that Hölder stability with respect to nonlinear perturbations can be characterized through stability with respect to linear ones. This is relevant in connection with constructing sequences of local solutions as the observation made earlier remains valid when nonlinear perturbations are replaced by linear ones. 
	\begin{corollary}\label{thm.L_vs_NL}
		Let $\lambda\in(0,+\infty)$. The  following statements are equivalent.
		\begin{itemize}
			\item[(i)]  There exists $\kappa>0$ such that, for any $\xi\in X^*$,
			\begin{align*}
				x\in \argmin_{y\in \B(\bar x,\delta)} \big\{f(y) - \langle \xi, y\rangle\big\}  \quad\implies \quad\|x-\bar x\| \le \kappa \|\xi\|^{\lambda}.
			\end{align*}
			
			\item[(ii)]  There exists $\kappa>0$ such that, for any Lipschitz function $\zeta:X\to\mathbb R$, 
			\begin{align*}
				x\in \argmin_{y\in \B(\bar x,\delta)}\big\{f(y) + \zeta (y)\big\} \quad\implies\quad\|x-\bar x\| \le \kappa\hspace*{0.03cm}\big(\Lip\,\zeta\big)^{\lambda},
			\end{align*}
			where $\Lip\,\zeta$ denotes the Lipschitz constant of $\zeta$.
			\item[(iii)]  There exists $\kappa>0$ such that, for any convex  function $\varphi: X \to \R\cup\{+\infty\}$, 
			\begin{align*}
				x\in \argmin_{y\in \B(\bar x,\delta)}\big\{f(y) + \varphi(y)\big\} \quad\implies\quad\|x-\bar x\| \le \kappa\hspace*{0.02cm} |\nabla \varphi|(\bar x)^{\lambda},
			\end{align*}
			where $| \nabla \varphi |(\bar x)$ denotes the local metric slope of $\varphi$ at $\bar x$.
		\end{itemize}
	\end{corollary}

	\begin{remark}
	   The results of Corollary \ref{thm.L_vs_NL} apply verbatim to the \L{}ojasiewicz-type inequality in part (iii) of Theorem \ref{thm.growth}.
	\end{remark}
    
	Applying the above result, one can provide a straightforward proof of the convergence of the proximal point algorithm. This is because each point in the sequence can be viewed as a minimizer of a convex perturbation of the original problem.
	\begin{corollary}\label{thm.proximal}Let $p, q \in (1, + \infty)$ be such that $p^{-1} + q^{-1}=1$. 
		Let the growth condition (\ref{growth.condi}) be fulfilled with $\gamma>0$. Let  $x_0 \in \mathbb{B}(\bar x, \delta)$ and $\varepsilon\in(0,\gamma)$. Let $\{x_k\}_{k\in\mathbb N}\subset X$ be a sequence satisfying
		\begin{align}
			x_{k} \in \argmin_{y\in \B(\bar x, \delta)}\left\lbrace f(y) + \frac{\varepsilon}{p}\| y-x_{k-1}\|^p \right\rbrace\quad \forall k\in\mathbb N.
		\end{align}
		The following assertions hold.
		\begin{itemize}
			\item[(i)] $x_k$ is a local minimizer of $f + \displaystyle\frac{\varepsilon}{p}\|\cdot - x_{k-1}\|^p$ for all $k\in\mathbb N$.
			
			\item[(ii)]  $x_k\longrightarrow \bar x$ as $k\longrightarrow +\infty$ and 
			\begin{equation}\label{proxi.x}
				\|x_k-\bar x\| \le \big(\gamma^{-1} \varepsilon\big)^{\frac{kq}{p}}\|x_0-\bar x\| \quad \forall k\in \N.
			\end{equation}
			\item[(iii)]  $f(x_k)\longrightarrow f(\bar x)$ as $k\longrightarrow +\infty$ and
			\begin{equation}\label{proxi.f}
				f(x_k)-f(\bar x)\le \big(\gamma^{-1}\varepsilon\big)^{kq} \big(f(x_0)-f(\bar x)\big) \quad \forall k \in \N.
			\end{equation}
		\end{itemize}
	\end{corollary}

    {
	For Algorithm \ref{algo} to be effective, the initial point must be sufficiently close to a local minimizer; the main difficulty lies in identifying such a point without knowing the exact solution. This requirement represents the gap between the theoretical result and its practical implementability. In practice, Corollary~\ref{thm.proximal} only guarantees the existence of sequences generated by Algorithm~\ref{algo} that converge to a local minimizer with quantitative rates.
    }
	\begin{algorithm}\label{algo}
		\caption{Proximal point algorithm}
		\begin{algorithmic}[1]
			\Require Initial point $x_0 \in X$,  parameter $\varepsilon > 0$,   number of iterations $K \in \mathbb{N}$
			\State Initialize $x \longleftarrow x_0$
			\For{$k \gets 0$ to $K-1$}
			\State Update $x \longleftarrow$ a local minimizer of $f + \frac{\varepsilon}{p} \|\cdot-x\|^p$
			\EndFor
			\State \Return $x$
		\end{algorithmic}
	\end{algorithm}
	\subsection{Literature review}
	The local growth condition we analyze has previously been characterized in terms of the metric subregularity of subdifferential mappings. The first result in this direction appears in \cite[Theorem 4.3]{ZT_1995}. Subsequently, Aragón and Geoffroy characterized the metric subregularity of the subdifferential of convex functions with local growth; see \cite[Theorem 3.3]{AG_2008} and \cite[Theorem 2.1]{AG_2014}. This result was later extended to nonconvex functions using the limiting subdifferential in \cite[Theorem 3.1]{DMN_2014} and \cite[Theorem 3.4]{MO_2015}; see also \cite[Theorem 4.1]{ZN_2015} for an analogous result concerning the Clarke-Rockafellar subdifferential.
	Subsequent results in particular settings include the characterization for semi-algebraic functions \cite[Theorem 3.1]{DI_2015}, and for functions that can be represented as the sum of a generalized twice differentiable function and a subdifferentially continuous, prox-regular, twice epi-differentiable function \cite[Theorem 4.1]{CQT_2022}. Further characterizations involving second-order notions of subdifferentiability are given in \cite[Theorem 4.1]{WS_2016} and \cite[Theorem 3.2]{CHN_2021}.  In Subsection \ref{subsectionmetric}, we present several results concerning the relationship between growth conditions and the subregularity of the Fenchel-Moreau subdifferential for potentially nonconvex functions. Our arguments differ from those in \cite[Theorem 3.3]{AG_2008}. The main result we provide is that if the Fenchel-Moreau subdifferential mapping has a dense image, then the global subregularity of the subdifferential is equivalent to a global growth condition. {In Remark \ref{gremark}, we compare the strong metric subregularity of the Fenchel–Moreau subdifferential with other commonly used subdifferentials, such as the limiting and Clarke subdifferentials.}
	\smallbreak
	Concerning the relationship between \L ojasiewicz inequalities and local growth conditions,~\cite[Corollary 2]{VT_2009} proves that a generalized \L ojasiewicz inequality (where the gradient is replaced by the Fréchet subdifferential) implies a local growth condition. For convex functions,  \cite[Section 4.2]{BolteLoja} and \cite[Theorem 5]{BNPS_2017} establish the equivalence between the classic Łojasiewicz inequality and a local growth condition. More recently, the relationship between the \L ojasiewicz inequality and the metric subregularity property of the subdifferential has been considered in \cite[Proposition 3.1]{PL_2019}. In Subsection \ref{subloj}, we present a result relating global growth conditions to a version of the Polyak-\L ojasiewicz condition, where the gradient at a point is replaced by the set of perturbations that have the point as a minimizer. 
	\smallbreak
	The proximal point algorithm was originally developed in \cite{R_1976} for solving inclusions (generalized equations) of monotone operators. In \cite{P_2002}, Pennanen showed that the convergence of the proximal point algorithm could be guaranteed locally, provided the iterations start sufficiently close to a solution and the graph of the operator is maximally monotone in the vicinity. More recently, Rockafellar provided a convergence result in \cite{R_2019} using a nonsmooth notion of second-order sufficiency, which was termed variational convexity. We provide a proof of convergence for a proximal point algorithm (Algorithm \ref{algo} below), which illustrates how the results on nonlinear perturbations developed in this paper can be used to derive convergence estimates from a simple conceptual principle.


	\section{Preliminaries} \label{prems}
	
	\subsection{Notation}	
	Let $\big(X,\|{\cdot}\|\big)$ be a normed space over the real numbers and consider a proper function $f: X \to \R \cup \{ + \infty \}$, i.e, $\mathrm{dom} \, f \neq \emptyset$.
	\smallbreak
	The continuous dual of $X$ is denoted by $X^*$, and the duality pairing between $X$ and $X^*$ is denoted by $\langle\cdot,\cdot\rangle:X^*\times X\to\mathbb R$.  We make the slight abuse of notation of denoting $\|{\cdot}\|$ also for the dual norm. We write $\mathbb B(x, r)$ for the closed ball of radius $r\in(0,+\infty)$ centered at a point $x \in X$; we set $\mathbb{B}\big(x, + \infty\big) = X$. Analogously, we write $\mathbb{B}_{*}(\xi, r)$ for the closed ball in the dual space $X^*$ of radius $r$ centered at $\xi\in X^*$. Whenever we write $c>0$, we mean $c\in(0,+\infty)$, that is, we implicitly assume $c$ is finite. 
	The closed convex hull of a set $A\subset X$ is denoted by $\conv A$. 
	\smallbreak
	For $p\in[1,+\infty)$, the duality mapping $J_p:X\rightrightarrows X^*$ is given by 
	\[
	J_p(x) := \left\{ \xi \in X^* : \langle \xi, x \rangle = \|\xi\| \|x\| \,\,\,\text{and}\,\,\,\|\xi\|=\|x\|^{p-1} \right\}.
	\]	
	Observe that when $X$ is a Hilbert space and  $p = 2$, the duality mapping reduces to the identity mapping. Given a nonempty set $C\subset X$, its indicator function $\iota_C:X\to \mathbb R\cup\{+\infty\}$  is given by
	\begin{align*}
		\iota_{C}(x):= \begin{cases}
			0 & if\,\, \,\, x\in C\\
			+\infty & if \,\, \,\,x\notin C
		\end{cases}.
	\end{align*}
	The \textit{Fenchel-Moreau subdifferential} of $f:X\to\mathbb R\cup \{+\infty\}$ at point $x\in X$ is given by
	\begin{align}\label{subdifferential}
		\partial f(x):= \left\lbrace \xi\in X^*:\, \text{$x$ is a minimizer of ${f}-\xi$}\right\rbrace. 
	\end{align}
	Observe that the convex subdifferential of a nonconvex function can have an empty domain. In this section, and in the whole sequel, we only refer to this notion of subdifferential. 
	\smallbreak 
	The function $f: X \to \R \cup \{ + \infty \}$ is said to be coercive if for every $c>0$ there exists $\delta>0$ such that 
	\begin{align*}
		f(x) \ge c, \text{ for every } x\in X \setminus\mathbb B\big(0,\delta\big).
	\end{align*}
	We say that $f$ is superlinear if for every $c >0$ there exists $\delta>0$ such that 
	\begin{align*}
		f(x) \ge c\| x \|, \text{ for every } x\in X\setminus\mathbb B\big(0,\delta\big).
	\end{align*}
	We see that if $f$ is superlinear, then it must be coercive. 
	Notice that if $\dom f$ is bounded, then $f$ is automatically superlinear. 
	\smallbreak

	Given a function $\zeta:X\to \mathbb R$, its Lipschitz constant is given by 
	\begin{align*}
		\Lip\, \zeta:= \sup_{\substack{x, y \in X \\ x \neq y}} \frac{|f(x) - f(y)|}{\| x - y \|}.
	\end{align*}
	We say that $\zeta$ is Lipschitz if $\Lip\,\zeta < +\infty$.
	Given a proper function $\varphi:X\to \mathbb R\cup\{+\infty\}$ and $x\in X$,  the local metric slope of $\varphi$ at $x$ is given by
	\begin{align*}
		| \nabla \varphi |(x) := \begin{cases}
			\displaystyle \limsup_{y \to x} \frac{\max \{0, \varphi(x) - \varphi(y)\}}{\| x - y \|} & \text{if $x\in\dom\,\varphi$}\\
			+\infty & \text{if $x\notin\dom\,\varphi$}. 
		\end{cases}
	\end{align*}
	We notice that if $\dom\,\varphi=X$ and  $\Lip\,\varphi<+\infty$, then its local slope is finite everywhere in $X$ and $|\nabla\varphi|(x)\le \Lip\,\varphi$ for all $x\in X$.

	\subsection{Set--valued mappings} Let $X, Y$ be normed spaces and consider a set--valued mapping $\mathcal{F} : X \rightrightarrows Y$. The domain of $\mathcal{F}$ is denoted by $\dom \, \mathcal{F} := \{ x \in X: \mathcal{F}(x) \neq \emptyset \}$. For a subset $A \subset X$, we will write
	\begin{align*}
		\mathcal{F}(A) := \{ y \in Y: \text{ there exists $x \in A$ such that $y \in \mathcal{F}(x)$} \}.
	\end{align*}
	The image of $\mathcal F$ is given by $\im \, \mathcal{F} := \mathcal{F}(X)$. The inverse of  $\mathcal F$ is the set-valued mapping $\mathcal{F}^{-1}:Y \rightrightarrows X$ given by
	\begin{align*}
		\mathcal F^{-1}(y) : = \left\lbrace x\in X:\, y \in\mathcal F(x)\right\rbrace.
	\end{align*}
	The mapping $\mathcal F$ is said to be \textit{lower hemicontinuous} at $\bar x\in X$ if for every open set $V \subset Y$ such that $\mathcal F(\bar x)\cap V\neq\emptyset$ there exists an open neighborhood $U\subset X$ of $\bar x$ such that $\mathcal F(x)\cap V\neq\emptyset$ for all $x\in U$. 
	We recall now the notion of $\lambda$-metric  subregularity of a set--valued mapping, see \cite[Section 4]{CiDo}, \cite[Section 12]{D_2021} and \cite[Section 3I]{DR_2009} for more details.
	\begin{definition}
		We say that $\mathcal F:X\rightrightarrows Y$ is strongly $\lambda$-subregular at $\bar x\in\mathcal F^{-1}(0)$ with parameters  $\alpha\in(0,+\infty]$ and $\kappa \in(0,+\infty)$  such that
		\begin{align*}
			\|x-\bar x\|\le \kappa d(0,\mathcal F(x))^\lambda\quad   \forall x\in \mathbb B(\bar x,\alpha),
		\end{align*}
		where $d(0, \mathcal{F}(x)) := \inf \left\{ \| y \| : y  \in \mathcal{F}(x) \right\}$.
	\end{definition}
	
	\subsection{Radon-Nikod\'ym property} To end this section, we give a short review of sets with the Radon-Nikod\'ym property (RNP) and we refer to the specialized book of Bourgin~\cite{GeoRNP} for more details. Let us begin recalling the definition of dentability, see~\cite[Definition 6.3.2]{VarTech} or~\cite[Section 1.7.1]{Penot}.
	\begin{definition}\label{RNP}
		Let $X$ be a Banach space and $\mathcal W$ a nonempty subset of $X$.
		\begin{enumerate}
			\item[$(i)$] We say that $\mathcal W$ is dentable if for every $\varepsilon>0$ there exist $\delta>0$ and $\xi \in X^*$ such that $\mathrm{diam} \, S(\mathcal W,\xi,\delta)<\varepsilon$,
			where the slice $S(\mathcal{W}, \xi, \delta)$ is defined by
			\begin{align*}
				S\big(\mathcal W,\xi,\delta\big):=\{\,v\in\mathcal W: \, \langle \xi, v \rangle < \inf_{w\in\mathcal W} \langle \xi,  w \rangle +\delta\}.
			\end{align*}	
			\item[$(ii)$] We say $\mathcal W$ has the RNP if every bounded subset of $\mathcal W$ is dentable.
		\end{enumerate}
	\end{definition}
	There are many different characterizations of the RNP, see, e.g., the book \cite{HuffGeoRNP}. We mention that the family of spaces having the RNP  includes all reflexive spaces, see~\cite[Corollary 4.1.5]{GeoRNP}, in particular the $L^p$-spaces for $p\in(1,\infty)$.  There are nevertheless interesting spaces without the RNP, such as $L^1([0,1])$, see~\cite[Example 2.1.2]{GeoRNP}. However, sets might possess the property even if the underlying space does not have it. For example, it is known that weakly compact convex subsets of Banach spaces have the RNP, see~\cite[Theorem 3.6.1]{GeoRNP}.
	\smallbreak
	One of the main applications of sets with RNP is the so-called Stegall's variational principle, see \cite[Theorem 1.153]{Penot}; we will employ the following version of this principle. 
	\begin{lemma}\label{LemSte}
		Let $X$ be a  Banach space and $f:X\to\mathbb R\cup\{+\infty\}$ be a proper lower semicontinuous function bounded from below. Suppose that 
		\begin{itemize}
			\item[(i)] $\conv(\dom f)$ has the RNP;
			
			\item[(ii)] $f$ is superlinear.
		\end{itemize}
		Then, there exists a dense set $\Theta\subset X^*$ such that for every $\xi\in \Theta$, $f-\xi$ has a unique minimizer.
	\end{lemma}
	
	\begin{proof}
		By superlinearity of $f$, the function $f-\xi$ is bounded from below for any $\xi\in X^*$, therefore $\dom\,f^* = X^*$, where $f^*$ is the Fenchel conjugate of $f$. We can then employ implication $(v)\implies (iii)$ of  \cite[Theorem 5]{Lassonde} to conclude the result.
	\end{proof}
	Observe that assumption $(i)$ of Lemma \ref{LemSte} is trivially satisfied when the underlying space is reflexive.  We also notice that assumption $(ii)$ of Lemma \ref{LemSte} is automatically satisfied  when the domain of the function is bounded.

	\subsection{Technical lemmas}
	In this subsection, we always assume that $X$ is a Banach space and $h: X \rightarrow \R \cup \{ + \infty \}$  a proper lower semicontinuous function. We fix a global minimizer $\bar x\in X$ of the function $h$. 
	\begin{lemma}\label{lem.Z_set}
		Let $\kappa > 0$. Consider the set
		\begin{align*}
			Z := \bigcup_{x \in X} \left\{ \xi \in X^*: \xi \in \partial h(x) \text{ and } \Vert x - \bar x \Vert \leq \kappa \Vert \xi \Vert^{q/p} \right\}.
		\end{align*}
		Then, for any $\theta \in (0, 1]$, it holds
		\begin{align*}
			h(x) \geq h(\bar x) + \dfrac{1}{2} \left( \dfrac{\theta}{2 \kappa} \right)^{p/q} \Vert x - \bar x \Vert^p \,\,\,\,\text{for every } x \in \bar x + 2\kappa\theta^{-1}J_p^{-1}(\overline{Z}),
		\end{align*}
		where $\overline{Z}$ denotes the closure of $Z$. 
	\end{lemma}
	
	\begin{proof}
		We note that $0\in Z$ since $\bar x\in (\partial h)^{-1}(0)$. Let $x \in \bar x + 2\kappa\theta^{-1}J_p^{-1}(\overline{Z})$. By definition of the duality mapping, there exists $\xi \in \overline{Z}$ such that  
		\begin{equation}\label{Hbthr}
			\Vert \xi \Vert = \left( \dfrac{\theta}{2 \kappa} \right)^{p - 1} \Vert x - \bar x \Vert^{p - 1} \text{ and } \langle \xi, x - \bar x \rangle = \Vert \xi \Vert \Vert x - \bar x \Vert.
		\end{equation}
		Let $\{\xi_n\}_{n\in\mathbb N}\subset Z$ be a sequence converging to $\xi$.  By definition of $Z$, for each $n\in \mathbb{N}$, there exists $x_n\in X$ such that $\xi_n\in\partial h(x_n)$ and $\|x_n-\bar x\| \le \kappa \|\xi_n\|^{q/p}$.  Now, since $h(\bar x) = \inf_{y\in X} h(y)$, one has
		\begin{align*}
			h(\bar x)-\langle \xi_n, x_n \rangle \le h(x_n)- \langle \xi_n, x_n \rangle = \inf_{y\in X} \{h(y)- \langle \xi_n, y \rangle \}\le h(x)- \langle \xi_n, x \rangle.
		\end{align*}
		Hence, for every $n \in\mathbb{N}$, we get
		\begin{align*}
			h(\bar x)- h(x)&\le \langle \xi_{n}, x_n - x \rangle = \langle \xi_{n}, x_n- \bar x  \rangle + \langle \xi_{n}, \bar x-x \rangle\\
			&\le \|\xi_n\|\|x-\bar x\| + \langle \xi_{n}, \bar x-x \rangle\le \kappa \|\xi_n\|^q + \langle \xi_n, \bar x-x \rangle.
		\end{align*}
		Letting $n \to \infty$, we have $h(\bar x)- h(x)\le \kappa \|\xi\|^q + \langle \xi, \bar x-x \rangle$. Using the identities in~\eqref{Hbthr},
		\begin{align*}
			h(\bar x) - h(x) & \le \kappa \|\xi\|^{q}-\|\xi\| \|x- \bar x\| = \frac{\theta^{p}}{2^p\kappa^{p-1}}\|x-\bar x\|^p -  \Big(\frac{\theta}{2\kappa}\Big)^{p-1}\|x-\bar x\|^p\\
			& =\left(\frac{\theta}{2\kappa}\right)^{p-1}\left(\dfrac{\theta}{2}-1\right)\|x-\bar x\|^p \le - \dfrac{1}{2}\Big(\frac{\theta}{2\kappa}\Big)^{p-1}\|x-\bar x\|^p.
		\end{align*}
		Therefore, we obtain  $h(x)-h(\bar x)\ge  2^{-1}(2\kappa)^{p - 1}\theta^{p-1}\|x-\bar x\|^p.$ Since $p/q=p-1$, the result follows.
	\end{proof}

	\begin{lemma}\label{lojalem1}
		Assume that $\conv(\dom\,h)$ has the RNP and  $\sqrt[\leftroot{1}\uproot{4}p]{h(\cdot)-h(\bar x)}$ is superlinear.
		For any $\upsilon \in (1,+\infty)$ and $x \notin(\partial h)^{-1}(0)$, there exists a sequence $\{\xi_n\}_{n\in\mathbb N}\subset X^*\setminus\{0\}$ satisfying 
		\begin{itemize}
			\item[$(i)$] $\displaystyle\argmin_{y\in X}\left\lbrace 	\sqrt[\leftroot{1}\uproot{4}p]{h(y)-h(\bar x)} - \langle \xi_{n}, y \rangle \right\rbrace\neq\emptyset$ for every $n\in \N$;
			
			\item[$(ii)$]  $\bar x \notin \displaystyle\argmin_{y\in X}\left\lbrace 	\sqrt[\leftroot{1}\uproot{4}p]{h(y)-h(\bar x)} - \langle \xi_{n}, y \rangle \right\rbrace$ for every $n\in\N$;
			
			\item[$(iii)$] $ \|\xi_n\| \to \upsilon \| x-\bar x\|^{-1} \displaystyle\sqrt[\leftroot{1}\uproot{4}p]{h(x)-h(\bar x)}$ as $n \to +\infty$.
		\end{itemize}
	\end{lemma}
	\begin{proof}
		Let $\sigma:X \to \R \cup\{+\infty\}$ be given by $\sigma(y):= \displaystyle\sqrt[\leftroot{1}\uproot{4}p]{h(y)- h(\bar x)}$. Let $x \notin (\partial h)^{-1}(0)$ and $\upsilon \in(1,+\infty)$. By the Hahn--Banach Theorem, there exists $\eta\in X^*$ such that 
		\begin{align}\label{HanhB}
			\langle \eta, x-\bar x \rangle = \|\eta\| \|x-\bar x\| \quad \text{and}\quad \|\eta\| = \frac{\upsilon \sigma (x)}{\|x-\bar x\|}.
		\end{align}
		Applying Lemma \ref{LemSte}, we can find a sequence $\{\eta_n\}_{n\in\mathbb N} \subset X^*$ such that
		\begin{align}\label{gel1}
			\argmin_{y\in X}\left\lbrace \sigma(y) - \langle \eta_n, y \rangle \right\rbrace\neq\emptyset\quad \text{and}\quad \|\eta_n-\eta\| \le \frac{1}{n}\qquad \forall n\in\mathbb N.
		\end{align}
		For $n$ large enough, we check that the sequence $\{ \eta_n \}_{n\in\mathbb N}$ satisfies the properties $(i)$--$(iii)$ in Lemma~\ref{lojalem1}. 
		We first observe that $\{ \eta_n \}_{n\in\mathbb N}$ is nonzero for $n \ge N_1$ large enough since the sequence converges to a nonzero element $\eta \neq 0$. For  $n \in \N$, by~\eqref{gel1}, there exists $\hat x_n \in X$ such that 
		\[ \displaystyle \hat x_n\in \argmin_{y\in X}\left\lbrace \sigma(y) - \langle \eta_n, y \rangle \right\rbrace.\]

		\noindent We show that there exists $N_2 \in\N$ such that $\hat x_{n}\neq \bar x$ for all $n\ge N_2$. If not,  there would exists a subsequence $\{\hat x_{n_k}\}_{k\in\N}$ such that $\hat x_{n_k}=\bar x$ for all $k\in\mathbb N$. Now, since each $\hat x_{n_k}=\bar x$ minimizes $\sigma - \eta_{n_k}$ and $\sigma(\bar x)=0$, we get
		\begin{align*}
			- \langle \eta_{n_k},\bar x \rangle \le \sigma (x) - \langle \eta_{n_k},  x \rangle, 
		\end{align*}
		which implies $\langle \eta_{n_k}, x-\bar x \rangle \le \sigma(x)$. Thanks to~\eqref{HanhB}, we obtain 
		\begin{align*}
			\upsilon \sigma(x) = \langle \eta, x-\bar x\rangle \le & ~ \sigma(x) + \langle \eta - \eta_{n_k}, x - \bar x \rangle \\
			\le & ~ \sigma(x) +  \|\eta_{n_k}-\eta\|\|x-\bar x\| \\
			\le & ~  \sigma(x)+ \frac{1}{n_k}\|x-\bar x\|.
		\end{align*}
		Sending $k \rightarrow \infty$, we get $\upsilon\sigma(x)\le\sigma(x)$, which makes a contradiction since $\sigma(x) > 0$ and $\upsilon > 1$.
		\noindent Define $\xi_{n} := \eta_{n + \max\{N_1, N_2\}}$ for each $n\in\N$. Thanks to the above observations, the sequence $\{ \xi_n \}_{n\in\mathbb N}$ satisfies the required properties in Lemma~\ref{lojalem1}.  
	\end{proof}
	Observe that the boundedness of $\dom\, f$ would be enough to satisfy the superlinearity assumption in the previous lemma. 
	\begin{lemma}\label{lojalem2}
		For any $x \in X$ and $\xi \in X^*$, if 
		\[
		x\in \argmin_{y\in X}\left\lbrace \sqrt[\leftroot{1}\uproot{4}p]{h(y)-h(\bar x)} - \langle \xi, y \rangle \right\rbrace,
		\]
		then $x \in \argmin_{y\in X}\left\lbrace h(y) -  p\sqrt[\leftroot{1}\uproot{4}q]{h(x)-h(\bar x)} \langle \xi, y \rangle  \right\rbrace.$
	\end{lemma}
	
	\begin{proof}It suffices to consider $x \not\in (\partial h)^{-1}(0)$; in this case  $h(x) > h(\bar x)$. Let $\varphi:X \to \R \cup \{+\infty\}$ be given by $\varphi(y):= {h(y) - h(\bar x)}$ for every $y \in X$. Note that $\varphi \geq 0$ since $\bar x$ is a global minimizer of $h$. Let  $y \in X$; we first suppose that $\varphi(y) < \varphi(x)$, using the Mean Value Theorem, we can find $t \in (0, 1)$ and $w = \varphi(y) + t(\varphi(x) - \varphi(y))$ such that
		\begin{align*}
			\sqrt[\leftroot{1}\uproot{4}p]{\varphi(x)} - \sqrt[\leftroot{1}\uproot{4}p]{\varphi(y)} = \frac{1}{p} w^{\frac{1 - p}{p}} (\varphi(x) - \varphi(y)).
		\end{align*}
		Since the function $s \mapsto s^{(1 - p)/p}$ is decreasing in $[0, + \infty)$ and $1/q = (p - 1)/p$, we infer that $w^{\frac{1 - p}{p}} \geq \frac{1}{\sqrt[\leftroot{1}\uproot{4}q]{\varphi(x)}}$. We conclude then that
		\begin{equation}\label{concavity}
			\frac{1}{p\sqrt[\leftroot{-1}\uproot{2}q]{\varphi(x)}}\big(\varphi(x)-\varphi(y)\big)\le \sqrt[\leftroot{1}\uproot{4}p]{\varphi(x)} - \sqrt[\leftroot{1}\uproot{4}p]{\varphi(y)}.
		\end{equation}
		Similarly, we can consider $y \in X$ such that $\varphi(y) \geq \varphi(x)$ and deduce the same estimate. Since $x$ minimizes $\sqrt[\leftroot{1}\uproot{4}p]{\varphi} - \xi$, we have 
		\[ \sqrt[\leftroot{1}\uproot{4}p]{\varphi(x)} - \sqrt[\leftroot{1}\uproot{4}p]{\varphi(y)} \le  \langle \xi, x -  y \rangle \quad \text{for every }y\in X.
		\]
		Combining this with \eqref{concavity} yields
		\begin{align*}
			\varphi(x) - \varphi(y) \le 	p\sqrt[\leftroot{1}\uproot{4}q]{\varphi(x)} \langle \xi, x-y \rangle\quad \text{for every } y\in X,
		\end{align*}
		which implies that $x \in \argmin_{y\in X}\{h(y)- p\sqrt[\leftroot{1}\uproot{4}q]{\varphi(x)}\,\langle \xi, y \rangle \}$.
	\end{proof}
	
	\noindent As a direct consequence of Lemma~\ref{lojalem1} and Lemma~\ref{lojalem2}, we obtain the following result.
	
	\begin{corollary}\label{coroloja}
		Assume that $\conv(\dom\,h)$ has the RNP and  $\sqrt[\leftroot{1}\uproot{4}p]{h(\cdot)-h(\bar x)}$ is superlinear.
		For any $\upsilon \in(1,+\infty)$ and $x\notin(\partial h)^{-1}(0)$, there exist sequences $\{\hat x_n\}\subset X\setminus\{\bar x\}$ and $\{\xi_n\}_{n\in\mathbb N}\subset X^*\setminus\{0\}$ satisfying the following properties, 
		\begin{itemize}
			\item[$(i)$] $p\sqrt[\leftroot{1}\uproot{4}q]{h(\hat x_n)-h(\bar x)}\,\xi_n\in \,\partial h(\hat x_n)$ for every $n\in\N$;
			
			\item[$(ii)$] $ \|\xi_n\| \to \upsilon \| x-\bar x\|^{-1} \displaystyle\sqrt[\leftroot{1}\uproot{4}p]{h(x) - h(\bar x)}$ as $n\to +\infty$.
		\end{itemize}
	\end{corollary}
	The previous corollary is inspired in the approach used in the proof of  \cite[Theorem 2]{VT_2009}. There they employ Ekeland's principle to the $p$-square of the functions; in our proof, we employ a different device since we do not count with the fuzzy rule for the Fr\'echet subdifferential to transform the Lipschitz perturbation of  Ekeland's principle into a linear one. 
	
	\section{Proofs}\label{mainchar}
	\subsection{Proof of Theorem~\ref{thm.growth}} We divide the proof into three parts: the first covers the straightforward implications, while the remaining parts address the nontrivial ones. We point out that the second part is not necessary, but we give it because the argument is of  independent interest and can be adapted to a more general setting, see Theorem~\ref{thm.grow.subreg} in the next section.
	\smallbreak 
	\noindent \textbf{Part 1.} $(i)\implies(ii)\implies(iii)$.
	We fix $x \in \mathbb{B}(\bar x, \delta)$ and $\xi \in X^*$ such that    
	\begin{align*}
		x \in \argmin_{y \in \B(\bar x, \delta)} \{ f(y) - \langle \xi, y \rangle \}.
	\end{align*}
	This implies that $f(x) - f(\bar x) \leq \langle \xi, x - \bar x \rangle$. From growth condition~\eqref{growth.condi}, we infer that
	\begin{equation}\label{growth.01}
		\gamma\|x-\bar x\|^p\le {f}(x)-{f}(\bar x)\le \langle \xi, x-\bar x \rangle\le \| \xi\|\| x - \bar x \|..
	\end{equation}
	Using the above estimate, we get $\|x-\bar x\|\le \kappa \|\xi\|^{\frac{1}{p-1}}$ with $\kappa:=1/\gamma^{\frac{1}{p-1}}$; noting that $q/p=1/(p-1)$ gives $(i)\implies(ii)$. Finally, observe that 
	\begin{align*}
		{f}(x)-{f}(\bar x)\le \langle \xi, x-\bar x \rangle \le \| \xi\|\| x - \bar x \| \le \kappa \|\xi\|^{q},
	\end{align*}
	since $1 + q/p = q$; taking $\mu := \kappa$ yields $(ii)\implies(iii)$,
	\smallbreak
	\noindent \textbf{Part 2.}  $(ii)\implies(i)$. Denote $f_{\delta}(y) = f(y) +  \iota_{\B(\bar x, \delta)}(y)$ for every $y \in X$. Since $X$ is reflexive, the set $\conv(\dom\,f_\delta)$ has the RNP. Moreover, since $\dom f_\delta$ is bounded, $f_\delta$ is superlinear.
	Applying {Lemma~\ref{LemSte}} to the function $f_\delta$, we know that the image of $\partial f_{\delta}$ is dense in $X^*$, in symbols, $\overline{\im(\partial f_{\delta})} = X^*$. Due to condition~\eqref{subreg}, we have that for any $\xi \in \im(\partial f_{\delta})$, there exists $x\in (\partial f_{\delta})^{-1}(\xi)$ such that $	\| x - \bar x \| \leq \kappa \| \xi \|^{\frac{q}{p}}.$
	Set
	\begin{align*}
		Z_{\delta} := \bigcup_{x \in X} \left\{ \xi \in X^*: \xi \in \partial f_{\delta}(x) \text{ and } \Vert x - \bar x \Vert \leq \kappa \Vert \xi \Vert^{q/p} \right\}.
	\end{align*}
	We see that $\im(\partial f_{\delta}) \subset Z_{\delta}$ and hence  $X^* = \overline{Z}_{\delta}$, so $J_p^{-1}(\overline{Z}_{\delta}) = X$. Applying Lemma~\ref{lem.Z_set} to the function $h = f_{\delta}$ and $Z = Z_{\delta}$, we conclude that for any $\theta \in (0, 1]$, 
	\begin{align*}
		f_{\delta}(x) \geq f_{\delta}(\bar x) + \dfrac{1}{2} \left( \dfrac{\theta}{2 \kappa} \right)^{p/q} \Vert x - \bar x \Vert^p \quad \text{for every } x \in X.
	\end{align*}
	Using definition of $f_{\delta}$ and choosing $\theta = 1$, we obtain that 
	\begin{align*}
		f(x) \geq f(\bar x) + \dfrac{1}{2} \left( \dfrac{1}{2 \kappa} \right)^{p/q} \Vert x - \bar x \Vert^p\quad \text{for every } x \in \B(\bar x, \delta). 
	\end{align*}
	It is then enough to set $\gamma:=2^{-p}\kappa^{-\frac{p}{q}}$.
	\smallbreak
	\noindent \textbf{Part 3.} $(iii)\implies(i)$. Let $x \in \dom~f \cap \big(\B(\bar x, \delta) \setminus \{ \bar x \}\big)$ and let $\upsilon \in (1,+\infty)$. Set $f_{\delta}$ as  in the previous part. 
    Applying Corollary~\ref{coroloja} to the function $h = f_{\delta}$ and $x \not\in (\partial f_{\delta})^{-1}(0)$, we can find sequences $\{ \hat x_n \} \subset X \setminus \{ \bar x \}$ and $\{ \xi_n \} \subset X^* \setminus \{ 0 \}$ satisfying
	\begin{equation}\label{eqn.approx_seq}
		p\sqrt[\leftroot{1}\uproot{4}q]{f_{\delta}(\hat x_n)-f_{\delta}(\bar x)}\, \xi_n \in \partial f_{\delta}(\hat x_n)\,\,\,\forall n\in\mathbb N\,\,\,   \text{ and }\,\,\,\|\xi_n\| \to  \upsilon\frac{\sqrt[\leftroot{1}\uproot{4}p]{f_\delta(x)-f_\delta(\bar x)}}{ \|x-\bar x\|}.
	\end{equation}
	We first observe that $\hat x_n \in \dom~f \cap \B(\bar x, \delta)$. Furthermore, thanks to~\eqref{eqn.approx_seq} and the  definition of $f_{\delta}$, we infer that
	\begin{align*}
		\hat x_n \in \argmin_{y \in \B(\bar x, \delta)} \left\{ f(y) -  p\sqrt[\leftroot{1}\uproot{4}q]{f(\hat x_n)-f(\bar x)} \langle \xi_n, y \rangle \right\}.
	\end{align*}
	Using condition~\eqref{PL.ineq}, we get
	\begin{align*}
		f(\hat x_n) - f(\bar x) \leq p^q \mu  (f(\hat x_n) - f(\bar x)) \| \xi_n \|^q,
	\end{align*}
	which implies that $1 \leq p^q \mu \| \xi_n \|^q.$
	Letting $n \to \infty$ and using~\eqref{eqn.approx_seq}, we are led to  
	\begin{align*}
		\| x - \bar x \|^p \leq p^q \mu \upsilon^q (f(x) - f(\bar x)). 
	\end{align*}
	Letting $\upsilon \searrow 1$, we obtain  inequality~\eqref{growth.condi} with $\gamma = \frac{1}{p^q\mu}$. 
	\subsection{Proof of Corollary~\ref{thm.L_vs_NL}}
	\begin{proof}
		Let $p\in(0,+\infty)$ be such that $\lambda=q/p$.
		For any $\xi \in X^*$, we first observe that $\| \xi \| = \Lip(\xi) = | \nabla \xi |(\bar x)$. And so, it is straightforward to check implications $(ii) \Longrightarrow (i)$ and $(iii) \Longrightarrow (i)$. 
		\smallbreak
		\textit{$(i)\implies(ii)$.} We assume that $(i)$ is satisfied. Thanks to Theorem~\ref{thm.growth}, the function satisfies a growth condition around $\bar x$, that is, there exists $\gamma > 0$ such that
		\begin{align*}
			f(x) \geq f(\bar x) + \gamma \| x - \bar x \|^p, \text{ for every } x \in \mathbb{B}(\bar x, \delta).
		\end{align*}
		Let $\zeta: X \to \R$ be a Lipschitz function and $x \in \mathbb{B}(\bar x, \delta)$ such that
		\begin{align*}
			x \in \argmin_{y \in \mathbb{B}(\bar x, \delta)} \{ f(y) + \zeta(y) \}.
		\end{align*}
		This leads to
		\begin{align*}
			f(x) + \zeta(x) = \min_{y \in \mathbb{B}(\bar x, \delta)} \{ f(y) + \zeta(y) \} \leq  f(\bar x) + \zeta(\bar x).
		\end{align*}
		Applying the growth condition, we are led to
		\begin{align*}
			\gamma \| x - \bar x \|^p \leq f(x) - f(\bar x) \leq \zeta(\bar x) - \zeta(x) \leq \Lip(\zeta) \| x - \bar x \|,
		\end{align*}
		which implies that $\| x - \bar x \| \leq \gamma^{-q/p}(\Lip(\zeta))^{\frac{q}{p}}.$
		\smallbreak
		\textit{$(i)\implies(iii)$}. The proof can be done analogously as above.
		Let $\varphi:X\to \mathbb R\cup\{+\infty\}$ be a convex proper function and $x\in X$ such that $	x \in \argmin_{y \in \mathbb{B}(\bar x, \delta)} \{ f(y) + \varphi(y) \}.$ 	Without the loss of generality, we assume $x \in\dom\,\varphi\setminus\{\bar x\}$. Then, 
		\begin{align*}
			\gamma\|x-\bar x\|^p\le \varphi(\bar x) - \varphi(x)\le \frac{\max\{0,\varphi(\bar x) - \varphi(x)\}}{\|\bar x-x\|} \|x-\bar x\|
		\end{align*}
		Notice that for a proper convex function, the local slope is equal to the global slope \cite[Theorem 2.4.9]{AGS_2008}. Hence, the local slope of $\varphi$ at $\bar{x}$ can be determined by
		\begin{align*}
			| \nabla \varphi |(\bar x) = \sup_{y \neq \bar x} \dfrac{\max \{0, \varphi(\bar x) - \varphi(y)\}}{\| \bar x - y \|}.
		\end{align*}
		Whence the result follows. 
	\end{proof}

	\subsection{Proof of Theorem~\ref{thm.proximal}}

	\begin{proof}
		Recall that $\gamma > 0$ is the constant given in Theorem \ref{thm.growth}.	
		Assume that $\varepsilon<\gamma$. 
		For each $k\in\mathbb N$, define $\zeta_k :X \to \mathbb R$ by
		\begin{align*}
			\zeta_{k}(y):= \frac{\varepsilon}{p}\|y-x_{k-1}\|^p.
		\end{align*}
		For each $k \in \N$, the local slope of $\zeta_k$ at $\bar x$ can be explicitly determined by $|\nabla\zeta_k|(\bar x)= \varepsilon\|x_{k-1}-\bar x\|^{p - 1}$. 
		Observe that for each $k\in\mathbb N$, we have 
		\begin{align}\label{ismin}
			x_{k}\in \argmin_{y\in \mathbb B(\bar x, \delta)}\left\lbrace f(y) + \zeta_k(y)\right\rbrace.
		\end{align}
		The convergence $x_k \xrightarrow{k \to \infty}  \bar x$ and $f(x_k) \xrightarrow{k \to \infty} f(\bar x)$ will follow directly from the estimates~\eqref{proxi.x} and~\eqref{proxi.f}. 
		Item $(i)$ will also follow from this convergence, as for $k\in\mathbb N$ large enough, $x_k$ must be in the interior of $\mathbb B(\bar x,\delta)$ and hence it will be local minimizer of $f + \dfrac{\varepsilon}{p}\|\cdot - x_{k-1}\|^p$. 
		\smallbreak
		$(ii)$
		Employing Theorem \ref{thm.L_vs_NL} and (\ref{ismin}), we find that 
		\begin{align*}
			\|x_k-\bar x\| \le \kappa |\nabla\zeta_k|(\bar x)^{\frac{q}{p}}=\kappa \varepsilon^{\frac{q}{p}} \|x_{k-1}-\bar x\|,
		\end{align*}
		where $\kappa:=\gamma^{-\frac{q}{p}}$, see Remark \ref{rmk.const-rela}.
		Thanks to the previous recursion, we are led to
		\begin{align*}
			\|x_k-\bar x\|\le (\kappa \varepsilon^{\frac{q}{p}})^k\|x_0-\bar x\|= (\gamma^{-1}\varepsilon)^{\frac{qk}{p}}\|x_0-\bar x\|.
		\end{align*}
		Since, $\gamma^{-1}\varepsilon<1$, we obtain $x_k\to \bar x$.
		\smallbreak
		$(iii)$
		Repeating  the argument as in Part 1 of the proof of Theorem \ref{thm.growth} in the same fashion as $(iii)\implies (i)$ of Corollary \ref{thm.L_vs_NL}, 
		we arrive at the estimate
		\begin{equation}\label{proxi.01}
			f(x_k) - f(\bar x) \leq \mu |\nabla\zeta_k|(\bar x)^{q} = \mu \varepsilon^q \|x_{k-1} - \bar x\|^p,
		\end{equation}
		where $\mu:=\gamma^{-\frac{q}{p}}$, see Remark \ref{rmk.const-rela}. Additionally, due to the growth condition, 
		\begin{equation}\label{proxi.02}
			\| x_{k - 1} - \bar x \|^p \leq  \gamma^{-1} \big(f(x_{k - 1}) - f(\bar x)\big).
		\end{equation}
		Combining \eqref{proxi.01} and \eqref{proxi.02}, we obtain
		\begin{align*}
			f(x_k) - f(\bar x) \leq  \gamma^{-q} \varepsilon^{q} (f(x_{k - 1}) - f(\bar x)). 
		\end{align*}
		This recursive relation leads to
		\begin{align*}
			f(x_k)-f(\bar x)\le \big(\gamma^{-1}\varepsilon\big)^{qk}\big(f(x_0)-f(\bar x)\big).
		\end{align*}
		Finally, since $\gamma^{-1}\varepsilon<1$, we obtain $f(x_k)\to f(\bar x)$ as $k \to \infty$.
	\end{proof}
	

	\section{Further discussion}\label{further}
	Let us recall that the \textit{Fenchel-Moreau subdifferential} of a (not necessarily convex) function $f:X\to\mathbb R\cup \{+\infty\}$ at point $x\in X$ is given by
	\begin{align*}
		\partial f(x):= \left\lbrace \xi\in X^*:\, \text{$x$ is a minimizer of ${f}-\xi$}\right\rbrace. 
	\end{align*}
	We will now discuss both local and global characterizations relating the previous subdifferential with growth conditions. 
	\subsection{Subregularity of the  Fenchel-Moreau subdifferential}\label{subsectionmetric}
	Lemma~\ref{lem.Z_set} gives a set of directions for which the growth condition is satisfied. We are going to state some instances in which one can conclude that this set is a neighborhood or even the whole space. A careful analysis of the proof of Lemma \ref{lem.Z_set} reveals that reflexivity of the underlying space and the lower  semicontinuity of the function are superfluous assumptions. 
	\begin{theorem}\label{thm.grow.subreg}
		Let $X$ be real normed space, $f: X \to \R \cup \{ + \infty \}$  a proper function and $\bar x \in (\partial f)^{-1}(0)$. The following assertions hold.
		\begin{itemize}
			\item[$(i)$]  If there exist $\beta \in (0, + \infty]$, $\kappa \in (0, + \infty)$ and a dense subset $\mathcal{W} \subset \mathbb{B}_{*}(0, \beta)$ such that for every $\xi \in \mathcal{W}$ there exists $x \in X$ so that 
			\begin{align*}
				\xi \in \partial f(x)\quad \text{and}\quad \| x - \bar x \| \leq \kappa \| \xi \|^{\frac{q}{p}}.
			\end{align*}
			Then, for any $\delta\in(0,+\infty)$, 
			\begin{align}\label{growthmax}
				{f}(x)\ge{f}(\bar x) +\frac{1}{2}\min\left\lbrace \frac{\beta}{\delta^{\frac{p}{q}}}, \frac{1}{(2\kappa)^{\frac{p}{q}}}\right\rbrace  \|x-\bar x\|^p\quad \text{for every $x\in \mathbb B(\bar x, \delta)$.}
			\end{align}

			\item[$(ii)$] Suppose that $\partial f: X\rightrightarrows X^*$ is $\frac{q}{p}$-subregular at $\bar x$ with parameters $\alpha\in(0,+\infty]$ and  $\kappa\in(0,+\infty)$. 
			Then, one has
			\begin{align*}
				{f}(x)\ge{f}(\bar x)+\frac{1}{2(2\kappa)^{\frac{p}{q}}}\|x-\bar x\|^p\quad \text{for every } x\in \bar x + 2\kappa\, J_{p}^{-1}\Big(\overline{\partial{f}\big(\mathbb B(\bar x,\alpha)\big)}\Big).
			\end{align*}  
		\end{itemize}
	\end{theorem}
	
	\begin{proof}$(i)$	Let $Z$ be the set in Lemma~\ref{lem.Z_set}. Clearly, $\mathcal W\subset Z$, and hence $\mathbb B_{*}\big(0,\beta\big)\subset \bar Z$. Since $J_p^{-1}\big(\mathbb B_{*}\big(0,\beta\big)\big) = \mathbb B\big(0,\beta^{\frac{q}{p}}\big)$, we conclude that, for any $\theta\in(0,1]$, 
		\begin{align}\label{growtheta}
			{f}(x)\ge{f}(\bar x)+\frac{1}{2}\Big(\frac{\theta}{2\kappa}\Big)^{\frac{p}{q}}\|x-\bar x\|^p\quad \forall x\in \mathbb B\Big(\bar x, \frac{2\kappa\beta^{\frac{q}{p}}}{\theta}\Big).
		\end{align}
		Let $\delta\in (0,2\kappa \beta^{\frac{q}{p}}]$, then choosing $\theta = 1$ in (\ref{growtheta}) yields (\ref{growthmax}). If $\delta\in (2\kappa \beta^{\frac{q}{p}},+\infty)$, we can choose $\theta = \delta^{-1}2\kappa\beta^{\frac{q}{p}}$ to conclude (\ref{growthmax}).
		\smallbreak
		$(ii)$ Again, let $Z$ be the set in Lemma~\ref{lem.Z_set}. Since $\partial f:X\rightrightarrows X^*$ is strongly metrically subregular at $\bar x$ with parameters $\alpha$ and $\kappa$, it follows that $\partial f\big(\mathbb B(\bar x,\alpha)\big)\subset Z$; whence the result can be concluded.
	\end{proof}
	
	\begin{remark}\normalfont
		One can easily find examples where functions  only satisfy a growth condition in a part of the domain. For example, consider $X = \R$ and $f(x) = \max \{ x, 0 \}^p$; the function $f$ satisfies a $p$-growth condition at $\bar x = 0$, but only in $(0, +\infty)$. Let  $X = \R^2$ and $g(x_1, x_2) = \max \{x_1, 0 \}^2 + x_2^2$; the function $g$ satisfies a quadratic growth condition at $(0,0)$, but only in $ [0, + \infty) \times \R$.
	\end{remark}
	
	In the following corollaries, we will show that the set $\mathcal{W}$ in Theorem~\ref{thm.grow.subreg} can be found and used in various scenarios.
	
	\begin{corollary}\label{corol.hemi-conti}
		Let $X$ be a real normed space, $f: X \to \R \cup \{ + \infty \}$ a proper function and $\bar x \in (\partial f)^{-1}(0)$. Suppose that $(\partial f)^{-1}: X^* \rightrightarrows X$ is lower hemicontinuous at zero. Then the following statements are equivalent.
		\begin{itemize}
			\item[(i)] There exist $\delta \in(0,+\infty]$ and $\gamma\in(0,+\infty)$ such that
			\begin{align*}
				{f}(x)\ge{f}(\bar x)+ \gamma\|x-\bar x\|^p, \text{ for every } x\in \mathbb B(\bar x, \delta).
			\end{align*} 
			
			\item[(ii)] The mapping  $\partial {f}: X \rightrightarrows X^*$ is strongly metrically $\frac{q}{p}$-subregular a $\bar x$.
		\end{itemize}
	\end{corollary}
	
	\begin{proof}
		\textit{$(i)\implies(ii)$}. The proof can be done analogously as Part 1 in the proof of Theorem~\ref{thm.growth}.  
		\smallbreak
		\textit{$(ii)\implies(i)$}.  Let $\alpha\in(0,+\infty]$ and $\kappa\in(0,+\infty)$	the parameters of strong subregularity of  $({\partial{f}})^{-1}:X^* \rightrightarrows X$.  By lower hemicontinuity,  there exists $\beta\in(0,+\infty]$ such that $({\partial{f}})^{-1}(\xi)\cap \mathbb B(\bar x,\alpha)\neq\emptyset$ for all $\xi\in \mathbb B_{*}(0,\beta)$. 
		This implies that for any $\xi\in \mathbb B_{*}(0,\beta)$ there exists $x\in \mathbb B(\bar x,\alpha)$ such that $\xi\in \partial {f}(x)$; using the subregularity of $\partial{f}:X\rightrightarrows X^*$, we conclude that $\|x-\bar x\|\le \kappa\|\xi\|^{\frac{q}{p}}$. 
		The hypothesis of Theorem \ref{thm.growth} are hence satisfied and the result follows.
	\end{proof}
	
	We now state another direct corollary of Theorem~\ref{thm.grow.subreg}, in which the equivalence between global growth and subregularity is given under the assumption that the image of the subdifferential is dense.
	
	\begin{corollary}\label{macabrecor}Let $X$ be a real normed space, $f: X \to \R \cup \{ + \infty \}$  a proper function and $\bar x \in (\partial f)^{-1}(0)$. Suppose that $\overline{\im\,\partial{f} }= X^*$. Then, the following statements are equivalent.
		\begin{itemize}
			\item[$(i)$] There exists $\gamma\in(0,+\infty)$ such that
			\begin{align*}
				{f}(x)\ge{f}(\bar x) + \gamma\|x-\bar x\|^p \quad  \text{for every } x\in X.
			\end{align*} 
			
			\item[$(ii)$] There exists $\kappa\in(0,+\infty)$ such that $\partial {f}$ is metrically $\frac{q}{p}$-subregular with parameters $\alpha=+\infty$ and $\kappa$.
		\end{itemize}
	\end{corollary}
	\begin{proof}
		The implication $(i)\implies(ii)$ is straightforward and follows the same argument given in the first part of the proof of Theorem \ref{thm.growth}. For the converse implication, let $\mathcal W:=\im\,\partial f$, and apply Theorem \ref{thm.grow.subreg}.
	\end{proof}
	
	Finally, to close the subsection, we give sufficient conditions for the image of the subdifferential of a function to be dense in the dual.

	\begin{proposition}\label{prop.clo_IM}
		Suppose that $X$ is a Banach space and that ${f}:X\to\mathbb R\cup\{+\infty\}$ is superlinear. Then the following statements hold.
		\begin{itemize}
			\item[(i)] If $X$ is reflexive and ${f}$ is weakly lower semicontinuous, then $\im \, \partial f = X^*$.
			
			\item[(ii)] If $\conv(\dom{f})$ has the RNP, then $\overline{\im \, \partial{f}} = X^*$.
		\end{itemize}
	\end{proposition}
	\begin{proof}
		$(i)$ We note that for each $\xi\in X^*$ the function $f-\xi$ is weakly lower semicontinuous and coercive, and hence it has at least one minimizer. One can directly obtain $(ii)$ from Lemma~\ref{LemSte}.
	\end{proof}

    {
    \begin{remark}\label{gremark}
        We point out that the strong metric subregularity of the Fenchel-Moreau subdifferential is related to stability of minimizers only. This not necessarily the same as critical points. To exemplify this, consider   $f:\mathbb R\to\mathbb R\cup\{+\infty\}$  given by
	\[ 
	f(x):=
	\begin{cases}
		2x^2 + x^3 \sin\left(\dfrac{1}{x^2} \right), & x \in [-1,1]\setminus\{0\}, \\
        0, & x=0\\
		+\infty, & \text{otherwise},
	\end{cases}
	\]
	This function $f$ satisfies the growth condition in Corollary \ref{macabrecor}$-(i)$ at $\bar x = 0$ and hence its Fenchel-Moreau subdifferential is strong metrically subregular at $\bar x$. 
    We see that there are infinitely many local minimizers close to $\bar x$. Consequently, any subdifferential for which local minimizers be critical points (such as the limiting, Fr\'echet or Clarke subdifferentials) cannot be strongly metrically subregular at $\bar x$.
    \end{remark}
    }
    
	\subsection{On the Polyak--\L{}ojasiewicz--type inequality}\label{subloj}
	In Theorem~\ref{thm.growth}, we have shown that a growth condition is equivalent to an analog of the so-called Polyak-\L ojasiewicz (replacing the gradient by the norm of linear perturbations).
	We will state a global version of this theorem in this subsection. 
	
	\begin{theorem}\label{thm.Loja}
		Let $X$ be a Banach space, $f: X \to \R \cup \{ + \infty \}$  a proper lower semicontinuous function and  $\bar x \in (\partial f)^{-1}(0)$.  Supposed that  $\conv (\dom \, f)$ is a bounded set with the RNP.
		Then, the following assertions are equivalent.
		\begin{itemize}
			\item[$(i)$] there exists $\gamma > 0$ such that
			\begin{align*}
				f(x) \geq f(\bar x) + \gamma \| x - \bar x \|^p\quad \text{for every } x \in X;
			\end{align*}
			\item[$(ii)$] there exists $\mu > 0$ such that
			\begin{align*}
				f(x) - f(\bar x) \leq \mu d(0, \partial f(x))^q \quad\text{for every } x \in X.
			\end{align*}
		\end{itemize}
	\end{theorem}
	\begin{proof} \textit{$(i)\implies(ii)$.} This is the same as in Part 1 of proof of Theorem~\ref{thm.growth}. 
		\smallbreak
		$(ii)\implies(i)$. The proof remains the same as in Part 3 of the proof of Theorem \ref{thm.growth}; the only differences are that the reflexivity of $X$ is not needed to establish that $\conv(\dom f)$ has the RNP, and that the superlinearity of $\sqrt[\leftroot{1}\uproot{4}p]{f(\cdot)-f(\bar x)}$ follows from the boundedness of the domain. 
	\end{proof}
	
	\section{An application to a tracking problem}\label{elliptic}

	In this section, we apply our results to an elliptic tracking problem without regularization terms. While we focus on a specific tracking problem for simplicity, the same arguments extend to a wide range of problems, including those constrained by parabolic PDEs or the Navier-Stokes equations.
	\smallbreak 
	The main result is that a second-order condition is equivalent to the stability of solutions to the problem with respect to the tracking data. 
	Although elliptic tracking problems have been extensively studied \cite{C_2012,CWW_2017,CWW_2018,J_2024,KNS_2017,PW_2018,QL_2024,QW_2018,W_2015}, this particular result has not been previously established. This finding is crucial for understanding the necessary assumptions to ensure the robustness of optimal solutions under data uncertainty. 
	\subsection{An elliptic tracking problem}\label{genSec}
	Let $d\in\{2,3\}$ and consider a bounded Lipschitz domain $\Omega\subset\mathbb R^d$. For $\alpha,\beta\in \mathbb R$ with $\alpha<\beta$, the feasible set is given by
	\begin{align*}
		\mathcal U:= \left\lbrace  u\in L^1(\Omega):\, u(x)\in[\alpha, \beta]\,\,\,\, \text{for a.e. $x\in\Omega$} \right\rbrace.
	\end{align*}
	For each $u\in\mathcal U$, there is a unique  associated state $y_u\in H_{0}^1(\Omega)\cap C(\bar\Omega)$ satisfying
	\begin{align}\label{varine}
		\int_{\Omega} \langle \nabla y_u, \nabla \varphi \rangle\, dx + \int_{\Omega} y^3_{u} \varphi \, dx = \int_{\Omega} u\varphi\, dx\quad \text{ for every } \varphi \in H_{0}^1(\Omega).
	\end{align}
	Formulation (\ref{varine}) says that $y_u$ is the weak solution of equation $-\Delta y + y_u^3 = u$. 
	\smallbreak
	Given a target $y_d\in L^2(\Omega)$, we consider the  elliptic tracking problem
	\begin{align*}
		(\mathcal P)\,\,\,\,\,\,	\min_{u\in\mathcal U} \left\lbrace\frac{1}{2} \int_{\Omega} | y_{u} - y_d |^2 dx \right\rbrace .
	\end{align*}
	The objective functional $\mathcal J:\mathcal U\to\mathbb R$ associated with problem $\mathcal P$ is given by $\mathcal{J}(u):= \frac{1}{2}\int_{\Omega}|y_u-y_d|^2\,dx$. 
	In order to give a good expression for its Gateaux differential, it is often convenient to introduce an adjoint state.  
	For each $u\in\mathcal U$, there is a unique associated costate $p_u\in H_0^1(\Omega)\cap C(\bar \Omega)$ satisfying
	\begin{align}\label{formulation}
		\int_{\Omega}\langle \nabla p_{u},\nabla\varphi \rangle \, dx + \int_{\Omega} 3y_u^2p_u \varphi dx = \int_{\Omega} (y_{u} - y_{d})\varphi dx \quad \text{ for every } \varphi\in H_0^1(\Omega).
	\end{align}
	This means that $p_u$ is the weak solution of equation $-\Delta p + 3y_u^2 p = y_u - y_d$. After straightforward calculations we obtain that $\nabla\mathcal J(u) = p_u $ for all $u\in\mathcal U$.
	\smallbreak
	From now on, we will fix a reference element ${\bar u}\in\mathcal U$ in the feasible set.
	\begin{lemma}\label{lem.1st-cond}
		If ${\bar u}\in\mathcal U$ is a minimizer of problem $\mathcal P$, then 
		\begin{align*}
			\int_{\Omega} \langle p_{{\bar u}}, u - {\bar u} \rangle\, dx \ge0 \quad \text{ for every } u\in\mathcal U. 
		\end{align*}
	\end{lemma}
	Lemma \ref{lem.1st-cond} condenses the first-order necessary condition. 
	Concerning second-order conditions, we refer to \cite[Theorem 2.4]{C_2012}. In order to give a useful expression for the second-order form of the objective functional, we introduce notation for linearized states. 
	\smallbreak
	For $v\in L^2(\Omega)$, we denote by $z_{v}\in H_0^1(\Omega)\cap C(\bar\Omega)$ the unique function satisfying
	\begin{align*}
		\int_{\Omega}\langle \nabla z_{v},\nabla\varphi \rangle \, dx + \int_{\Omega} 3y_{{\bar u}}^2z_v \varphi dx = \int_{\Omega} v\varphi dx \quad \text{for every } \varphi\in H_0^1(\Omega).
	\end{align*}
	This means that $z_v$ is the weak solution of $-\Delta z_v + 3y_{{\bar u}}^2 z_v = v$. Some straightforward calculations yield 
	\begin{align*}
		\mathcal J''({\bar u})v^2 = \int_{\Omega} \big(1 - 6 y_{{\bar u}}p_{{\bar u}}\big) z_{v}^2 dx \quad \text{ for every } v\in L^2(\Omega).
	\end{align*}
	We will employ the following characterization of growth for the objective functional. 
	\begin{proposition}\label{Prop}
		The following statements are equivalent. 
		\begin{itemize}
			\item[(i)] There exist numbers $\delta>0$ and $c>0$ such that 
			\begin{align*}
				\mathcal J'({\bar u})v + \frac{1}{2}\mathcal J''({\bar u})v^2\ge c\|z_v\|_{L^2(\Omega)}^2
			\end{align*}
			for all $v\in \mathcal U - {\bar u}$ satisfying $\|z_v\|_{L^2(\Omega)}\le\delta$.
			\smallbreak
			\item[(ii)] There exist numbers $\delta>0$ and $c>0$ such that 
			\begin{align*}
				\mathcal J(u) - \mathcal J({\bar u})\ge c \|y_{u}-y_{{\bar u}}\|_{L^2(\Omega)}^2
			\end{align*}
			for all $u\in\mathcal U$ satisfying $\|y_{u}-y_{{\bar u}}\|_{L^2(\Omega)}\le\delta$.
		\end{itemize}
	\end{proposition}
	\begin{proof}
		This is a straightforward application of Taylor's expansion formula, and the estimates in \cite[Section 3]{CDJ_2023}.
	\end{proof}
    {
	\begin{remark}
    We point out that the growth condition in Proposition-\ref{Prop}(i) can be satisfied for non--bang-bang controls, unlike growth conditions involving the control variable; see \cite[Section~2.1]{CWW_2018}. For example, if the tracking term $y_d \in L^2(\Omega)$ is chosen so that the optimal solution is a control $\bar u \in \mathcal U$ with $\bar u(x) \in (\alpha,\beta)$ for a.e.\ $x \in \Omega$ and $y_{\bar u}=y_d$, then $p_{\bar u}=0$, and the condition in Proposition-\ref{Prop}(i) is satisfied due to the form of the second variation.
	\end{remark}}
	\subsection{Sensitivity of minimizers with respect to the tracking data}\label{Mainrel}
	
	Consider the family $\{\mathcal P_{\eta}\}_{\eta\in L^2(\Omega)}$ of tracking problems
	\begin{align*}
		(\mathcal P_\eta)\,\,\,\,\,\min_{u\in\mathcal U} \left\lbrace \frac{1}{2}\int_{\Omega}\big|y_{u} - (y_d+\eta)\big|^2\, dx \right\rbrace.
	\end{align*}
	
	The next characterization is the main result of this section. 
	\begin{theorem}\label{THM}
		Suppose that ${\bar u}\in\mathcal U$ is a strict minimizer of the problem $(\mathcal{P})$.
		The following statements are equivalent. 
		\begin{itemize}
			\item[(i)] There exist numbers $\delta>0$ and $c>0$ such that 
			\begin{align*}
				\mathcal J'({\bar u})v + \frac{1}{2}\mathcal J''({\bar u})v^2\ge c\|z_v\|_{L^2(\Omega)}^2
			\end{align*}
			for all $v\in \mathcal U - {\bar u}$ satisfying $\|z_v\|_{L^2(\Omega)}\le\delta$.
			\smallbreak
			\item[(ii)] There exists a number $\kappa>0$ such that 
			\begin{align*}
				\|y_{u_\eta} - y_{{\bar u}}\|_{L^2(\Omega)}\le \kappa \|\eta\|_{L^2(\Omega)}
			\end{align*}
			for any  $\eta\in L^2(\Omega)$ and any minimizer $u_\eta\in\mathcal U$ of problem  $\mathcal P_\eta$.
		\end{itemize}
	\end{theorem}
	\begin{proof}
		Define the functional $f:L^2(\Omega)\to\mathbb R\cup\{+\infty\}$ by
		\begin{align*}
			f(y):=\begin{cases}
				\displaystyle \int_{\Omega}\Big( \frac{1}{2}y^2  + y_{d} y\Big)\, dx & \text{ if $y=y_u$ for some $u\in\mathcal U$},\\\\
				+\infty & \text{ otherwise}.
			\end{cases}
		\end{align*}
		By definition of $f$, we have the identity
		\begin{align}
			\mathcal J(u) - \mathcal J({\bar u})  = 	f(y_u) - f(y_{{\bar u}}) \quad \text{for every } u\in\mathcal U. 
		\end{align}
		Notice that $y_{{\bar u}}$ is a minimizer of $f$. Moreover note that for any $\xi\in L^2(\Omega)$ and $u\in\mathcal U$,
		\begin{align}\label{equiv0}
			\text{$y_u$ is a minimizer of $f-\xi$ if and only if $u$ is a minimizer of problem $\mathcal P_{y_{d}+\xi}$}.
		\end{align}
		We begin showing implication $(ii)\implies(i)$. 
		Let $\xi\in L^2(\Omega)$ and consider a minimizer $y = y_u$ of $f-\xi$. Then,  by (\ref{equiv0}) and the hypothesis, 
		\begin{align*}
			\| y_{u} - y_{{\bar u}} \|_{L^2(\Omega)} \le  \kappa \| \xi \|_{L^2(\Omega)}.
		\end{align*}
		By implication $(ii)\implies(i)$ of  Theorem \ref{thm.growth},  there exists $c>0$ such that 
		\begin{align*}
			\mathcal J(u) - \mathcal J({\bar u}) =  f(y_u) - f(y_{{\bar u}}) \ge c\|y_{u} - y_{{\bar u}}\|_{L^2(\Omega)}^2\quad\text{ for every } u\in\mathcal U. 
		\end{align*}
		The implication follows then from Proposition \ref{Prop}.
		\smallbreak
		Let us now prove implication $(i)\implies(ii)$. By Proposition \ref{Prop},  there exist $\delta>0$ and $c>0$ such that $\mathcal J(u) - \mathcal J({\bar u})\ge c\|y_{u}-y_{{\bar u}}\|_{L^2(\Omega)}^2$ for all $u\in\mathcal U$ with $\|y_{u}-y_{{\bar u}}\|_{L^2(\Omega)}\le\delta$. We claim that there exists $\varepsilon>0$ such that  $\mathcal J(u)-\mathcal J({\bar u})\ge \varepsilon$ for all $u\in\mathcal U$ such that $\|y_{u}-y_{{\bar u}}\|_{L^2(\Omega)}\ge \delta$. Otherwise, there would exist a sequence $\{u_n\}_{n\in\mathbb N}\subset\mathcal U$ satisfying $\|y_{u_n}-y_{{\bar u}}\|_{L^2(\Omega)}\ge \delta$ and $\mathcal J(u_n)\to \mathcal J({\bar u})$ as $n\to +\infty$. We could extract a  subsequence $\{u_{n_k}\}_{k\in\mathbb N}$ of $\{u_n\}_{n\in\mathbb N}$ converging weakly in $L^1(\Omega)$ to some $\hat u\in\mathcal U$. This would imply $\mathcal J(\hat u)=\mathcal J({\bar u})$, but since ${\bar u}$ is a strict minimizer, $\hat u = {\bar u}$, and consequently $y_{u_{n_k}}\to y_{{\bar u}}$ in $L^2(\Omega)$; a contradiction.  Set  $M:=\sup_{u\in\mathcal U}\| y_{u} - y_d \|_{L^2(\Omega)}$. Then, for all $u\in\mathcal U$ with $\|y_{u} - y_{{\bar u}}\|_{L^2(\Omega)}\ge\delta$,
		\begin{align*}
			\mathcal J(u) - \mathcal J(u)\ge \frac{\varepsilon}{M^2}\| y_{u} - y_{{\bar u}}\|_{L^2(\Omega)}^2.
		\end{align*}
		Setting $\gamma:=\min\{c,\varepsilon/M^2\}$, we see that 
		\begin{align*}
			f(y_{u}) - f(y_{{\bar u}}) = \mathcal J(u)-\mathcal J({\bar u})\ge \gamma\| y_{u} - y_{{\bar u}}\|_{L^2(\Omega)}^2\quad \text{for every } u\in\mathcal U.
		\end{align*}
		By implication $(i)\implies(ii)$ of Theorem \ref{thm.growth}, there exists $\kappa>0$ such that 
		\begin{align}\label{stabi}
			\| y_{u} - y_{{\bar u}} \|_{L^2(\Omega)}\le \kappa \|\xi\|_{L^2(\Omega)} 
		\end{align}
		for any $\xi\in L^2(\Omega)$ and  any minimizer $y_u\in L^2(\Omega)$ of $f-\xi$. The implication follows then by (\ref{stabi}).
	\end{proof}
	\begin{remark}
		In Theorem \ref{THM}, we have assumed that the minimizer is strict; while this is not always the case, one can prove that the set of targets that give a strict minimizer is generic in $L^2(\Omega)$. In some tracking problems, the existence of a strict minimizer (and hence a unique one) for every target can be achieved by making natural hypotheses on the data of the problem, see \cite[Corollary 5.3]{CW_2021}. There are also problems for which there is a tracking datum yielding more than one minimizer even with regularization terms, see \cite[Theorem 1.1]{P_2023}. In the case where the control is also tracked and the control system is not linear, there will always be targets yielding more than one minimizer, see \cite[Theorem 2.3]{CH_2022}.
	\end{remark}	
	\begin{remark}
		We remark that growth conditions with respect to the control can also be studied and characterized in terms of stability with respect to linear perturbations. Several papers have studied these relations in optimal control, we refer the interested reader to \cite{ANV_2025,DZZ_2023,JS_2024,JOV_2025,M_2024,OV_2020,QSV_2020}.
	\end{remark}
	\begin{remark}
	We point out that in Theorem \ref{THM}, the space where the equivalence $(i) \Leftrightarrow (ii)$ from Theorem \ref{thm.growth} is applied is $L^2(\Omega)$.
	\end{remark}


	\section*{Acknowledgments}
	The first author was supported by the Alexander von Humboldt Foundation, and the second one by the Austrian Science Fund (grant FWF P-36344N).
	We thank Eduardo Casas and Aris Daniilidis for their  comments and suggestions on the paper, and to Constantin Christof for insightful discussions on tracking problems. We are also grateful to the anonymous reviewers for their remarks. 
	\nocite{*}
	\bibliographystyle{siam}
	\bibliography{references}
\end{document}